\documentclass[11pt]{article}
\usepackage{amsfonts,amsmath,amssymb,amsthm,graphicx}
\usepackage{graphicx}
\usepackage{color}

\newtheorem{theorem}{Theorem}[section]

\newtheorem{lemma}[theorem]{Lemma}
\newtheorem{claim}[theorem]{Claim}
\newtheorem{cor}[theorem]{Corollary}

\newtheorem{conj}[theorem]{Conjecture}

\newcommand{\bproof}{\noindent{\bf Proof. }}
\newcommand{\eproof}{\hfill $\bullet$\\}

\newcommand{\rank}{\mbox{\rm rank }}
\newcommand{\real}{{\mathbb R }}
\newcommand{\rat}{{\mathbb Q}}

\newcommand{\tran}{{\mbox{td}}}

\newcommand{\scrs}{{{\mathcal S}}}

\newcommand{\mb}[1]{\mathbb{#1}}
\newcommand{\nib}[1]{\noindent {\bf #1}}

\newcommand{\sm}{\setminus}
\newcommand{\ov}{\overline}
\newcommand{\sub}{\subseteq}

\newcommand{\es}{\emptyset}
\newcommand{\eps}{\varepsilon}

\newcommand{\bsize}[1]{\left| #1 \right|}

\title{Global rigidity of 2-dimensional
direction-length frameworks}

\author{Katie Clinch\thanks{School of Mathematical Sciences,
Queen Mary, University of London, Mile End Road, London E1 4NS,
UK. Email: k.clinch@qmul.ac.uk} \and Bill Jackson\thanks{School
of Mathematical Sciences, Queen Mary, University of London, Mile End
Road, London E1 4NS, UK. Email: b.jackson@qmul.ac.uk} \and
Peter Keevash\thanks{Mathematical Institute, University of Oxford, Oxford, UK.
Email: keevash@maths.ox.ac.uk. Research supported in part by ERC Consolidator Grant 647678.} }

\date{2 July 2016}

\begin{document}
\maketitle

\begin{abstract}
A 2-dimensional direction-length framework is a collection of points
in the plane which are linked by pairwise constraints that fix the
direction or length of the line segments joining certain pairs of
points. We represent it as a  pair $(G,p)$, where $G=(V;D,L)$ is a
`mixed' graph and $p:V\to\real^2$ is a point configuration for $V$.
It is globally rigid if every direction-length framework $(G,q)$
which satisfies the same constraints can be obtained from $(G,p)$ by
a translation or a rotation by $180^\circ$. We show that the problem
of characterising when a generic  framework $(G,p)$ is globally rigid can be
reduced to the case when $G$ belongs to a special family of
`direction irreducible' mixed graphs, and
prove that {every} generic realisation of
a direction irreducible mixed graph $G$ is globally rigid if and
only if $G$ is 2-connected, direction-balanced and redundantly
rigid.
\end{abstract}

\section{Introduction}\label{sec:intro}

A finite configuration of points in Euclidean space with local
constraints may be informally described as globally rigid if the
constraints determine the point set up to congruence. It is a
fundamental open problem to give a nice characterisation of global
rigidity in various settings.
Our setting here is that of a $d$-dimensional {\em direction-length
framework}, which is a pair $(G,p)$, where $G=(V;D,L)$ is a `mixed'
graph and $p:V\to\real^d$ is a point configuration for $V$. (We will be particularly concerned with the case when $d=2$.)
We call
the graph $G$ {\em mixed} because it has two types of edges: we
refer to edges in $D$ as {\em direction edges} and edges in $L$ as
{\em length edges}. The graph may contain parallel edges as long as they are of different types.
Two  direction-length frameworks $(G,p)$ and
$(G,q)$ are {\em equivalent} if $p(u)-p(v)$ is a scalar
multiple of $q(u)-q(v)$ for all $uv\in D$ with $q(u)\neq q(v)$, and
$\|p(u)-p(v)\|=\|q(u)-q(v)\|$ for all $uv\in L$. Two point
configurations $p$ and $q$ for $V$ are {\em congruent} if either
$p(u)-p(v)=q(u)-q(v)$ for all $u,v\in V$, or $p(u)-p(v)=q(v)-q(u)$
for all $u,v\in V$. (Thus $p$ and $q$ are congruent if $p$ can be
obtained from $q$ by a translation, possibly followed by a
rotation by $180^\circ$.) A direction-length framework $(G,p)$ is {\em globally
rigid} if $p$ is congruent to $q$ for every framework $(G,q)$ which
is equivalent to $(G,p)$. It is {\em rigid} if there exists an
$\eps>0$ such that if a framework $(G,q)$ is equivalent to $(G,p)$
and satisfies $\|p(v)-q(v)\|<\eps$ for all $v\in V$ then $p$ is
congruent to $q$ (equivalently every continuous motion of the
vertices of $(G,p)$ which satisfies the direction and length
constraints given by the edges results in a framework $(G,q)$ with
$p$ congruent to $q$). The framework $(G,p)$ is {\em redundantly
rigid} if $(G-e,p)$ is rigid for all $e\in D\cup L$.

\begin{figure}
\vspace{-10cm}
\centering
\includegraphics[scale=0.8]{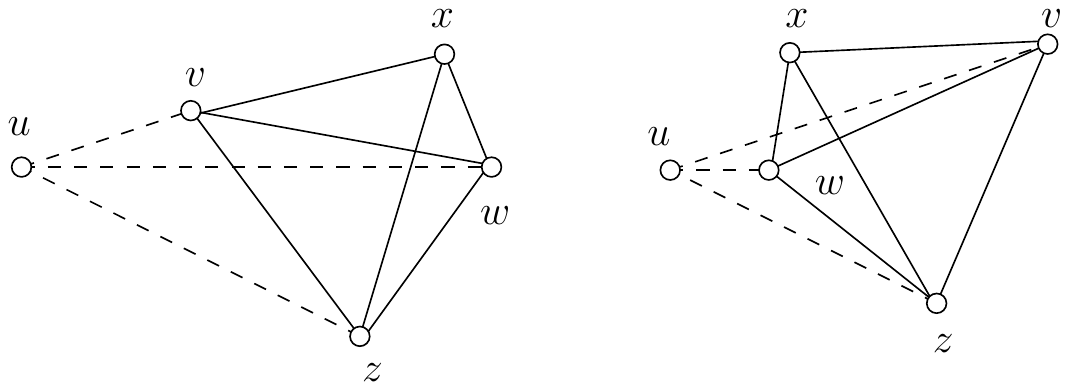}
\vspace{-10.5cm}
\caption{\label{fig1} Two equivalent but non-congruent direction-length frameworks. We use solid or dashed lines to indicate
length or direction constraints, respectively. The frameworks are
rigid but not globally rigid.}
\end{figure}

We will consider {\em generic} frameworks, meaning that the set
containing the coordinates of all of the vertices is algebraically
independent over the rationals; this eliminates many pathologies. It
follows from \cite{JJ2,JK1} that rigidity is a `generic property' in
the sense that if some realisation of a mixed graph $G$ as a generic framework in $\real^d$ is
rigid then all generic realisations of $G$ in $\real^d$ are rigid.
This implies that redundant rigidity is also a generic property and
allows us to describe a mixed graph $G$ as being {\em rigid} or {\em
redundantly rigid} in $\real^d$ if some (or equivalently if every)
generic realisation of $G$ has these properties. It is not known
whether global rigidity is a generic property (however this
statement would follow from Conjecture \ref{characterise} below in
the 2-dimensional case).

Both rigidity {and} global rigidity are known to be generic
properties for $d$-dimensional {\em pure frameworks}, i.e. frameworks which
contain only length constraints or only direction constraints. We will not
give formal definitions for such frameworks, but note that they are
similar to those for direction-length frameworks, except that the
notion of congruence has to allow not only translations, but also
rotations in the case of length-pure frameworks, and dilations in
the case of direction-pure frameworks.

The problems of characterising rigidity and global rigidity for
2-dimensional generic length-pure frameworks were solved by Laman
\cite{L} and  Jackson and Jord\'an \cite{JJ}, respectively. In
particular, \cite{JJ} proved that a $2$-dimensional generic length-pure framework $(G,p)$ is globally length-rigid if and only if
either $G$ is a complete graph on at most $3$ vertices, or $G$ is
$3$-connected and redundantly length-rigid. The problems of
characterising rigidity and global rigidity for $d$-dimensional
generic length-pure frameworks are open for $d\geq 3$. In contrast,
Whiteley \cite{W} showed that rigidity {and} global rigidity are
{\em equivalent} generic properties for direction-pure frameworks and
characterised the $d$-dimensional generic frameworks which have
these properties for all $d$.

Since direction-length frameworks
are more general than length-pure frameworks, we will henceforth restrict our attention
to direction-length frameworks of dimension two.
Rigid generic (2-dimensional) direction-length frameworks
were characterised by Servatius and Whitely in \cite{SW}.
They also showed that every generic realisation of a rigid mixed
graph with exactly one length edge is globally rigid. Further
results on  global rigidity were obtained by Jackson and Jord\'an in
\cite{JJ1} who showed that two necessary conditions for a generic
2-dimensional direction-length framework $(G,p)$ to be globally
rigid are that $G$ is 2-connected and {\em direction-balanced} i.e.
whenever $H_1,H_2$ are subgraphs of $G$ with $G=H_1\cup H_2$,
$V(H_1)\cap V(H_2)= \{u,v\}$ and $V(H_1)\sm V(H_2)\ne \es \ne
V(H_2)\sm V(H_1)$, both $H_1$ and $H_2$ must contain a direction
edge of $G$ distinct from $uv$. They also showed that these
conditions are sufficient when $G$ is redundantly rigid and has
$2|V|-1$ edges.

We will see in Section \ref{sec:rigid} that we may define a matroid
$M(G)$ on the edge set of a mixed graph $G$ in such a way that $G$
is rigid if and only if $M(G)$ has rank $2|V|-2$. The above
sufficient condition for global rigidity, that $G$ is redundantly
rigid and has $2|V|-1$ edges, is equivalent to the edge set of $G$
being a rigid circuit of $M(G)$. A mixed graph $G$ is redundantly rigid
if and only if it is rigid and every edge of  $G$ is contained in a
circuit of $M(G)$. We say that $G$ is $M$-connected if it
satisfies the stronger condition that every pair of edges of $G$ is
contained in a circuit of $M(G)$ i.e.
$M(G)$ is a connected matroid.  Clinch \cite{KC} has recently shown
that the above mentioned necessary conditions for generic global
rigidity are also sufficient when the underlying mixed graph is
$M$-connected and rigid.

\begin{theorem}\label{thm:katie} Suppose $(G,p)$ is a generic realisation of an $M$-connected rigid mixed graph $G$. Then $(G,p)$ is
globally rigid if and only if $G$ is $2$-connected and
direction-balanced.
\end{theorem}

Unfortunately, Clinch's result does not give a complete
characterisation of generic global rigidity because $M$-connectivity
is not a necessary condition for the global rigidity of generic rigid frameworks. This follows from
the above mentioned fact that every generic realisation of a
(minimally) rigid mixed graph with exactly one length edge is
globally rigid, or from the fact that global rigidity is preserved
if we join a new vertex to an existing globally rigid framework by
two direction constraints. (The underlying graphs in both
constructions are not even redundantly rigid.) We can generalise the
second construction as follows.


Suppose $(G,p)$ is a generic realisation of a mixed graph $G$ which has a
proper induced subgraph $H$ such that  the graph $G/H$ obtained from $G$ by contracting $H$ to a single
vertex (deleting all edges of $H$ but keeping all other edges of $G$, possibly as parallel edges) has only direction
edges and is the union of two edge-disjoint spanning trees. We will
see in Section \ref{sec:dirred} that
$G-e$ is not rigid for all direction edges $e$ which do not belong
to $H$ (hence $G$ is not redundantly rigid), and that $(G,p)$ is globally rigid
if and only if $(H,p|_H)$ is globally rigid.

These observations lead us to consider a more general reduction
operation for a mixed graph $G$. We say that $G$ admits a {\em
direction reduction} to a subgraph $H$ if either:
\begin{enumerate}
\item[(R1)] $H=G-e$ for some edge $e\in D$ which belongs to a direction-pure circuit in the rigidity matroid of
$G$, or
\item[(R2)] $H$ is a proper induced subgraph of $G$, and
$G/H$ is direction-pure and is the
union of two edge-disjoint spanning trees.
\end{enumerate}
If $G$ has no direction reduction, then we say that $G$ is {\em
direction irreducible}. (We will describe an efficient algorithm in
Section \ref{sec:con} which either finds a direction reduction of a
given mixed graph  or concludes that it is direction
irreducible.) An example of a direction reduction is given in Figure \ref{fig2}.

\begin{figure}
\centering
\input{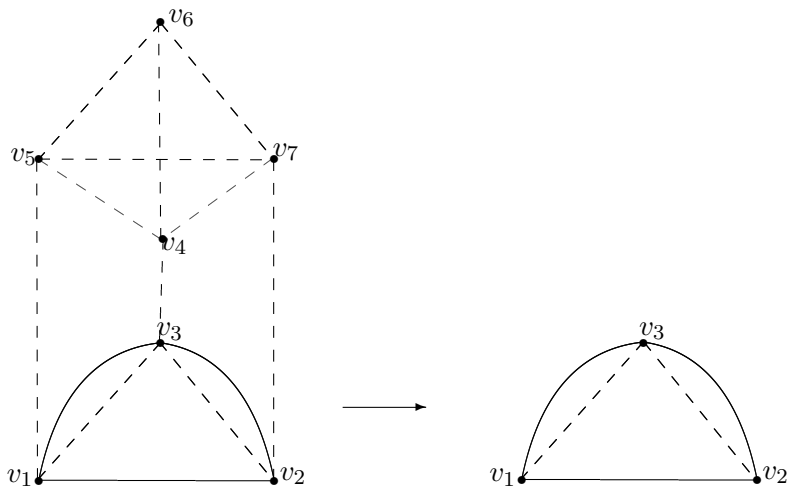}
\caption{The graph $G$ on the left is direction reducible to the subgraph $H$ on the right in two steps. Since the direction edge $v_5v_7$ is contained in the direction-pure circuit induced by $\{v_4,v_5,v_6,v_7\}$ we can delete $v_5v_7$ by (R1). The graph we now obtain by contracting $H$ to a single vertex is direction-pure and is the union of two edge-disjoint spanning trees so we can reduce $G$ to $H$ by (R2). Theorem \ref{reduce} now tells us that a generic framework $(G,p)$ is globally rigid if and only if $(H,p|_H)$ is globally rigid. Since $H$ is an $M$-connected, $(H,p|_H)$ is globally rigid by Theorem \ref{thm:katie}. Hence $(G,p)$ is globally rigid.}\label{fig2}
\end{figure}

Our first result reduces the problem of characterising the global
rigidity of a generic framework $(G,p)$ to the case when $G$ is
direction irreducible.

\begin{theorem}\label{reduce}
Suppose $(G,p)$ is a generic direction-length framework and $G$
admits a direction reduction to a subgraph $H$. Then $(G,p)$ is
globally rigid if and only if $(H,p|_H)$ is globally rigid.
\end{theorem}



We will obtain structural information about the family of direction
irreducible mixed graphs which are not redundantly rigid and use it
to prove our main result which, together with Theorem \ref{reduce}, gives a complete characterisation of mixed graphs with the property that all their generic realisations are globally rigid.

\begin{theorem}\label{redrigid}
Suppose $G$ is  a direction irreducible mixed graph with at least
two length edges. Then every generic realisation of $G$ is globally
rigid if and only if $G$ is $2$-connected, direction-balanced and
redundantly rigid.
\end{theorem}




The organisation of this paper by section is
 \ref{sec:intro}: Introduction,
\ref{sec:prelim}: Preliminaries,
\ref{sec:realise}: Realisations of graphs with given direction constraints,
\ref{sec:dirred}: Direction reduction,
\ref{sec:irreducible}: Direction irreducible graphs,
\ref{sec:mainresult}: Proof of Theorem \ref{redrigid},
\ref{sec:con}: Algorithmic considerations,
\ref{sec:close}: Closing Remarks.


\section{Preliminaries}\label{sec:prelim}

In this section we collect tools from diverse areas that we will use in our proofs.

\subsection{Rigidity}\label{sec:rigid}

Suppose $(G,p)$ is a 2-dimensional direction-length framework. Its
{\em rigidity matrix} is a $(|D|+|L|)\times 2|V|$ matrix $R(G,p)$,
where each edge in $D \cup L$ corresponds to a row and each vertex
in $V$ corresponds to a pair of consecutive columns. We choose an
arbitrary reference orientation for the edges, and use the notation
$e=uv$ to mean that $e$ has been oriented from $u$ to $v$. Fix an
edge $e$, a vertex $x$, and write $p(u)-p(v)=(a,b)$. Then the two
entries in the rigidity matrix corresponding to $e$ and $x$ are as
follows. If $e \in L$ we take $(a,b)$ if $x=u$, $(-a,-b)$ if $x=v$,
$(0,0)$ otherwise. If $e \in D$ we take $(b,-a)$ if $x=u$, $(-b,a)$
if $x=v$, $(0,0)$ otherwise.

We refer to vectors in the null space $Z(G,p)$ of $R(G,p)$ as {\em
infinitesimal motions}. The labeling of the columns of $R(G,p)$
allows us to consider each infinitesimal motion $z$ as a map from
$V$ to $\mb{R}^2$, with the properties that $z(u)-z(v)$ is
perpendicular to $p(u)-p(v)$ if $e=uv \in L$, or parallel to
$p(u)-p(v)$ if $e=uv \in D$.
For any $t \in \mb{R}^2$ the translation
given by $z(v)=t$ for all $v \in V$ is an infinitesimal motion, so
$\dim Z(G,p)\geq 2$ and $\rank R(G,p)\leq 2|V|-2$. We can `factor
out' translations by restricting attention to realisations $(G,p)$ that
are in {\em standard position}, i.e.\ satisfy $p(v_0)=(0,0)$ for
some fixed $v_0 \in V$. Write $R(G,p)_{v_0}$ for the matrix obtained
from $R(G,p)$ by deleting the $2$ columns corresponding to $v_0$ and
let $Z(G,p)_{v_0}$ be its null space. Then $Z(G,p)_{v_0}$ is
isomorphic to the subspace $Z(G,p)_{v_0}^*$ of $Z(G,p)$ consisting
of all infinitesimal motions which fix $v_0$. Since all non-zero
translations belong to $Z(G,p) \sm Z(G,p)_{v_0}^*$ we have $\dim
Z(G,p)_{v_0} = \dim Z(G,p) - 2$, so $\rank R(G,p)_{v_0} = 2|V|-2 -
\dim Z(G,p)_{v_0} = 2|V|-\dim Z(G,p) = \rank R(G,p)$. We say that
the framework $(G,p)$ is {\em infinitesimally rigid} if $\rank
R(G,p)=2|V|-2$, and is {\em independent} if the rows of $R(G,p)$ are
linearly independent.

A property P of frameworks is {\em generic} if whenever some generic
realisation of a graph $G$ has property P then all generic
realisations of $G$ have property P. If P is a generic property then
we say that a graph $G$ has property P if some generic realisation
of $G$ has property P (or equivalently all generic realisations of
$G$ have property P). Infinitesimal rigidity and independence are
both generic properties, as the rank of $R(G,p)$ is the same for all
generic realisations of $G$. Results from \cite{JJ2,JK1}, which will
be described in Section \ref{sec:dg}, imply that infinitesimal
rigidity and rigidity are equivalent properties for generic
direction-length frameworks. Thus rigidity and redundant rigidity
are also generic properties.

The rigidity matrix of $(G,p)$ defines the {\em rigidity matroid} of
$(G,p)$: the ground set $D\cup L$ corresponds to rows of the
rigidity matrix, and a subset is independent when the corresponding
rows are linearly independent. Any two generic realisations of $G$
have the same rigidity matroid, which we call the (2-dimensional)
{\em rigidity matroid} $M(G)$ of $G$. (We refer the reader to
\cite{O} for an introduction to the theory of matroids.)

Servatius and Whiteley \cite{SW} characterised independence in the
rigidity matroid of a mixed graph $G=(V;D,L)$: a set of edges
$F\subseteq D\cup L$  is independent in $M(G)$ if and only if for
all $\es \neq F' \subseteq F$ we have $|F'| \le 2|V(F')|-2$, with
strict inequality if $F'$ is length or direction-pure. This implies
that $F$ is a circuit of $M(G)$ if and only if $F-e$ is independent
for all $e\in F$ and either $F$ is mixed with $|F|=2|V(F)|-1$, or $F$
is pure with $|F|=2|V(F)|-2$.

Servatius and Whiteley also gave the following recursive
construction for independent rigid mixed graphs, i.e.\ bases in the
rigidity matroid of the `complete mixed graph'. A {\em
$0$-extension} of $G$ is a mixed graph obtained from $G$ by adding a
new vertex $v$ and two edges at $v$, either of which may be a length
edge or a direction edge, and which may go to the same vertex of $G$
if they consist of one length edge and one direction edge. A {\em
$1$-extension} of $G$ is a mixed graph obtained from $G$ by adding a
new vertex $v$, deleting an edge $e$ of $G$, and adding three edges
at $v$, such that the neighbours of $v$ include both endpoints of
$e$, neither $D$ nor $L$ decrease in size,
and two new edges may go to the same vertex if they are of different types.
They showed that
$0$-extensions and $1$-extensions preserve independence and
rigidity, and conversely, any independent rigid mixed graph can be
constructed starting from a single vertex by a sequence of
$0$-extensions and $1$-extensions.

\subsection{$M$-circuits and $M$-components}

It is well known that a matroid can be expressed as the direct sum
of its connected components, which are the equivalence classes of
the relation $\sim$, where $e \sim f$ if $e=f$ or there is a circuit
containing $e$ and $f$. We define the {\em $M$-components} of a
mixed graph $G=(V;D,L)$ to be the subgraphs induced by the edges in
the connected components of its rigidity matroid $M(G)$. Similarly,
we define the {\em $M$-circuits} of $G$ to be the subgraphs induced
by the edges of the circuits of $M(G)$.
We can use the direct sum decomposition of the rigidity matroid
$M(G)$ to calculate its rank, which we will denote by $r(G)$.
Indeed, if $G$ has $M$-components $H_1,\dots,H_m$ then we have $r(G)
= \sum_{i=1}^m r(H_i)$, where $r(H_i)$ is  $2|V(H_i)|-3$ when $H_i$
is pure
and is $2|V(H_i)|-2$ otherwise.
We can use this fact to show that $M$-connectivity is equivalent to
redundant rigidity when $G$ is direction irreducible and satisfies
the necessary conditions for generic global rigidity described in
Section \ref{sec:intro}.

\begin{lemma}\label{lem:redMcon}
Suppose $G$ is a direction irreducible, $2$-connected, direction-balanced mixed graph. Then $G$ is $M$-connected if and only if $G$ is redundantly rigid.
\end{lemma}

\bproof 
We have already seen, in Section \ref{sec:intro} that redundant rigidity is a necessary condition for $M$-connectedness. To prove sufficiency we
suppose that $G$ is redundantly rigid but not $M$-connected. Let  $H_1,H_2,\ldots ,H_m$ be
the $M$-components of $G$.  Let $V_i=V(H_i)$, $X_i=V_i -
\bigcup_{j\neq i} V_j$ and $Y_i=V_i-X_i$ for all $1\leq i\leq m$.
Since $G$ is redundantly rigid,  every edge of $G$ is contained in
some $M$-circuit. Hence  $|V_i|\geq 3$ for all $1\leq i\leq m$.
Since $G$ is 2-connected,  $|Y_i|\geq 2$ for
all $1\leq i\leq m$, and since $G$ is direction-balanced, $|Y_i|\geq 3$ when $H_i$ is
length-pure. Since $G$ is direction irreducible, no direction edge
of $G$ is contained in a direction-pure $M$-circuit. This implies
that each of the $M$-connected components is either mixed or
length-pure. Without loss of generality, we may assume that
$H_1, H_2, \ldots, H_\ell$ are length-pure for some $0\leq \ell \leq
t$, and $H_{\ell+1},H_{\ell+2},\ldots, H_m$ are mixed. Then
\begin{align*}
r(G) & = \sum_{i=1}^{\ell} (2|V_i|-3) + \sum_{i=\ell+1}^{m}(2|V_i|-2)\\
     & = \sum_{i=1}^{\ell} (2|X_i|+2|Y_i|-3) + \sum_{i=\ell+1}^{m}(2|X_i|+2|Y_i|-2)\\
     &\geq \sum_{i=1}^{m} (2|X_i|+|Y_i|),
\end{align*}
since $|Y_i|\geq 2$ for all $1\leq i\leq m$, with strict inequality
when $1\leq i\leq \ell$. Since the $X_i$ are all disjoint, we have
$\sum_{i=1}^{m}|X_i|=\left\vert \bigcup_{i=1}^{m} X_i \right\vert$.
Also, since each element of $Y_i$ is contained in at least one other
$Y_j$ with $j\neq i$, we have $\sum_{i=1}^{m}|Y_i|\geq
2|\bigcup_{i=1}^{t} Y_i|$. Thus
\[
r(G) \geq 2\left(\left\vert \bigcup_{i=1}^{m} X_i \right\vert
+\left\vert\bigcup_{i=1}^{m} Y_i \right\vert \right)
    = 2|V|.
\]
This contradicts the fact that $r(G)\leq 2|V|-2$.
\eproof

\subsection{Boundedness and global rigidity}\label{sec:bounded}

Now we recall some results from \cite{JK1,JK2}. A direction-length
framework $(G,p)$ is {\em bounded} if there exists a real number $K$
such that $\|q(u)-q(v)\|<K$ for all $u,v\in V$ whenever $(G,q)$ is a
framework equivalent to $(G,p)$. Our first result shows that the boundedness of $(G,p)$ is
equivalent to the rigidity of an augmented framework.

\begin{lemma}\label{augment}\cite[Theorem 5.1]{JK1}  Let $(G,p)$ be a direction-length framework and let $G^+$ be obtained from $G$ by adding a direction edge parallel to each length
edge of $G$.
Then  $(G,p)$ is bounded if and only if $(G^+,p)$ is rigid.
\end{lemma}

Lemma \ref{augment} implies that
boundedness is a generic property, and we say that a mixed graph $G$ is {\em bounded} if some, or equivalently every,
generic realisation of $G$ is bounded.
It also implies that every rigid mixed graph is bounded.

A mixed graph $G=(V;D,L)$ is {\em
direction-independent} if $D$ is independent in the direction-length
rigidity matroid of $G$, i.e.\ the rows of $R(G,p)$ corresponding to
$D$ are linearly independent for any generic $p$. The facts that
direction-pure $M$-circuits are direction-rigid
and that
direction-rigidity and global direction-rigidity are equivalent for direction-pure frameworks, allow us
to reduce the problem of deciding if a mixed graph is bounded to the
family of direction-independent mixed graphs.  The following
characterisation of boundedness for direction-independent mixed  graphs follows from \cite[Theorem 5.1 and Corollary 4.3]{JK1}.

\begin{lemma}\label{bounded}
Suppose that $G=(V;D,L)$ is a direction-independent mixed graph. Then $G$ is bounded if
and only if $G/L$ has two edge-disjoint spanning trees (where $G/L$
is the graph obtained from $G$ by contracting each edge in $L$ and keeping all multiple copies of direction edges created by this contraction).
\end{lemma}

A {\em bounded component} of $G$ is a maximal bounded subgraph of $G$.
It is shown in \cite{JK1} that each edge $e\in L$ lies in a bounded component
and that the vertex sets of the bounded components partition $V$.
The following lemma is implicit in \cite{JK1};
for completeness we include a short proof. We will need the well known
result of Nash-Williams \cite{NW} that the edge set of a graph $H$ can be covered by $k$ forests if and only if
every non-empty set $X$ of vertices of $H$ induces at most $k|X|-k$ edges of $H$.


\begin{lemma}\label{boundedcomps}
Suppose $G=(V;D,L)$ is direction-independent and $\scrs$ is a set of bounded
components of $G$ with $|\scrs|\geq 2$. Then there are at most
$2|\scrs|-3$ edges of $G$ joining distinct components in $\scrs$.
\end{lemma}

\bproof Suppose on the contrary that  there are
at least $2|\scrs|-2$ edges of $G$ that join distinct components
in $\scrs$. Suppose also that $\scrs$ is minimal with respect to this property (and the condition that $|\scrs|\geq 2$).
Let $G'=(V';D',L')$ be the subgraph of $G$ spanned by $\cup_{C_i \in \scrs} C_i$.
Let $H$ be a graph with vertex set  $\scrs$ and exactly $2|\scrs|-2$ edges,
each of which correspond to a distinct edge of $G$ joining two
components in $\scrs$.
The minimality of $\scrs$ implies that every non-empty set $X$ of vertices of $H$ induces at most $2|X|-2$ edges of $H$ and hence,
by the above mentioned result of Nash-Williams, $H$ can be partitioned into two edge-disjoint spanning trees.
By
Lemma \ref{bounded}, for each bounded component $C_i=(V_i;D_i,L_i) \in \scrs$, $C_i/L_i$ has
two edge-disjoint spanning trees. We can combine the edge sets of these trees with the edge sets of the two edge-disjoint  spanning trees of $H$
to obtain two edge-disjoint spanning trees in $G'/L'$. Lemma \ref{bounded} now implies that $G'$
is bounded and hence  is contained in a single bounded component of $G$. This
contradicts the fact that $|\scrs| \ge 2$. \eproof
Now we can state the main result of \cite{JK2} on global rigidity,
which establishes when length-redundancy is a necessary condition
for generic global rigidity and takes a first step towards
understanding when direction-redundancy is necessary. A subgraph of
a mixed graph is said to be {\em trivial} if it has exactly one
vertex, otherwise it is {\em non-trivial}.

\begin{theorem}\label{hend}\cite{JK2}
Suppose that $(G,p)$ is a globally rigid generic realisation of a
mixed graph $G=(V;D,L)$
and $e$ is an edge of $G$.\\
(a) If $e\in L$ and $|L|\geq 2$ then $G-e$ is rigid. \\
(b) If $e\in D$ and $G-e$ has a non-trivial rigid subgraph then $G-e$ is either rigid or unbounded.
\end{theorem}

\subsection{Substitution}\label{sec:sub}

The following subgraph substitution operation is an important tool
which we will use throughout this paper. Suppose $G=(V;D,L)$ is a
mixed graph, $U\subseteq V$, $H=G[U]$ is the  subgraph of $G$
induced by $U$, and $H'$ is another mixed graph with vertex set $U$.
Then the {\em substitution} $G'$ of $H$ by $H'$ in $G$ is obtained
from $G$ by deleting all edges of $H$ and adding all edges of $H'$.
We record the following properties.

\begin{lemma}\label{subrigid}
If $G$, $H$ and $H'$ are rigid then $G'$ is rigid.
\end{lemma}

\bproof
The ranks of $G$ and $G'$ are both equal to the rank of the graph obtained from $G$
by joining all pairs of vertices of $H$ by both a direction and a length edge.
\eproof

\begin{lemma}\label{replace}
Suppose $p:V\to \real^2$ is such that $(G,p)$ and $(H',p|_U)$ are both globally rigid.
Then $(G',p)$ is globally rigid.
\end{lemma}

\bproof
Let $(G',q)$ be an equivalent framework to $(G',p)$.
Since $(H',p|_U)$ is globally rigid, $q|_U$ is congruent to $p|_U$.
In particular, $(H,q|_U)$ and $(H,p|_U)$ are equivalent.
But $G$ and $G'$ agree on all edges not contained in $U$,
so $(G,q)$ and $(G,p)$ are equivalent. Since $(G,p)$ is globally rigid,
$q$ and $p$ are congruent. Hence $(G',p)$ is globally rigid.
\eproof

\subsection{Differential geometry and the framework space}\label{sec:dg}

Here we recall some basic concepts of differential geometry.
Let $X$ be a smooth manifold, $f:X\to \mb{R}^n$ be a smooth map, and
$k$ be the maximum rank of its derivative $df|_y$ over all $y\in X$.
A point $x\in X$ is a {\em regular point} of $f$ if $\rank
df|_x=k$. The Inverse Function Theorem states that if $U$ is open in
$\mb{R}^k$, $f:U \to \mb{R}^k$ is smooth, $x\in U$, and the
derivative $df|_x:\mb{R}^k \to \mb{R}^k$ is non-singular, then $f$
maps any sufficiently small open neighbourhood of $x$
diffeomorphically onto an open subset of $\mb{R}^k$. The following
lemma is a simple consequence of this (see \cite[Lemma 3.3]{JK2}).

\begin{lemma}\label{inverse}
Let $U$ be an open subset of $\mb{R}^m$, $f:U\to \mb{R}^n$ be a smooth map and
$x\in U$ be a regular point of $f$. Suppose that the rank of $df|_x$ is $n$.
Then there exists an open neighbourhood $W\subseteq U$ of $x$
such that $f(W)$ is an open neighbourhood of $f(x)$ in $\mb{R}^n$.
\end{lemma}

The following function plays an important role in rigidity theory.
For $v_1,v_2\in V$ with $p(v_i)=(x_i,y_i)$ let
$l_p(v_1,v_2)=(x_1-x_2)^2+(y_1-y_2)^2$,
and $s_p(v_1,v_2)=(y_1-y_2)/(x_1-x_2)$ whenever $x_1\neq x_2$.
Suppose $e=v_1v_2 \in D\cup L$.
We say that $e$ is {\em vertical} in $(G,p)$ if $x_1=x_2$.
The {\em length} of $e$ in $(G,p)$ is $l_p(e)=l_p(v_1,v_2)$,
and the {\em slope} of $e$ is $s_p(e)=s_p(v_1,v_2)$,
whenever $e$ is not vertical in $(G,p)$.
Let $V=\{v_1,v_2,\ldots,v_n\}$ and $D\cup L=\{e_1,e_2,\ldots,e_m\}$.
We view $p$ as a point $(p(v_1),p(v_2),\ldots,p(v_n))$ in $\real^{2n}$.
Let $T$ be the set of all points $p\in \real^{2n}$ such that $(G,p)$ has no vertical
direction edges. Then the {\it rigidity map} $f_G:T\to \real^m$ is
given by $f_G(p)=(h(e_1),h(e_2),\ldots,h(e_m))$, where
$h(e_i)=l_p(e_i)$ if $e_i\in L$ and $h(e_i)=s_p(e_i)$ if $e_i\in D$.

One can verify (see \cite{JK2}) that each row in the Jacobian matrix
of the rigidity map is a non-zero multiple of the corresponding row
in the rigidity matrix, so these matrices have the same rank. Thus
the rigidity matrix achieves its maximum rank at a realisation $(G,p)$
when $p$ is a regular point of the rigidity map. In particular, this
is the case when $(G,p)$ is generic.

The {\em framework space} $S_{G,p,v_0} \sub \mb{R}^{2|V|-2}$ consists of all $q$
in standard position with respect to $v_0$ with $(G,q)$ equivalent to $(G,p)$.
Here we recall that `standard position' means that $q(v_0)=(0,0)$,
and we identify a realisation $(G,q)$ with the vector in $\mb{R}^{2|V|-2}$ obtained by
concatenating the vectors $q(v)$ for $v \in V \sm \{v_0\}$.
The proof of the following lemma is the same as that of \cite[Theorem 1.3]{JK2},
omitting the part that proves $-p_0 \notin C$, as this is now an assumption.

\begin{lemma}\label{deletedirection}
Suppose $(G,p)$ is a generic direction-length framework, $e$ is a
direction edge of $G$, $G$ is rigid, and $H=G-e$ is bounded and not
rigid. Let $v_0$ be a vertex of $G$, let $p_0$ be obtained from $p$
by translating $v_0$ to the origin, and let $C$ be the connected
component of the framework space $S_{H,p,v_0}$ containing $p_0$.
Then $C$ is diffeomorphic to a circle. Furthermore, if $-p_0 \notin C$
then $(G,p)$ is not globally rigid.
\end{lemma}

\subsection{Field extensions and genericity}\label{sec:ex}

A mixed framework $(G,p)$ is {\em quasi-generic} if it is a translation of a generic framework.
We will be mostly concerned with quasi-generic frameworks in standard position,
i.e.\ with one vertex positioned at the origin. Such frameworks are characterised
by the following elementary lemma.

\begin{lemma}\cite{JJ2}\label{quasi}
Let $(G,p)$ be a framework with vertices $\{v_1,v_2,...,v_n\}$,
$p(v_1)=(0,0)$ and $p(v_i)=(p_{2i-1},p_{2i})$ for $2\leq i\leq n$.
Then $(G,p)$ is quasi-generic if and only if
$\{p_3,p_4,\ldots, p_{2n}\}$ is algebraically independent over $\rat$.
\end{lemma}

Given a vector $p\in \mb{R}^d$, $\mb{Q}(p)$ denotes the field extension of $\rat$ by the
coordinates of $p$. We say that $p$ is {\em generic} in $\mb{R}^d$ if the coordinates of $p$ are algebraically independent over $\rat$.
Given fields $K,L$ with $K\sub L$ the {\em
transcendence degree} $td[L:K]$ of $L$ over $K$ is the size of the
largest subset of $L$ which is algebraically independent over $K$.
A reformulation of Lemma
\ref{quasi} is that if $(G,p)$ is a framework with $n$ vertices, one
of which is at the origin, then $(G,p)$ is quasi-generic if and only
if $td[\mb{Q}(p):\mb{Q}]=2n-2$.

Recall that $G=(V;D,L)$ is independent if $D \cup L$ is independent
in the (generic) rigidity matroid of $G$, and that $f_G$ denotes the
rigidity map of $G$, which is defined at all realisations $(G,p)$ with
no vertical direction edges. The next result relates the genericity of  $f_G(p)$
to the genericity of  $p$ when $G$ is independent.

\begin{lemma}\cite{JJ2}\label{qgen1}
Suppose that $G$ is an independent mixed graph and
$(G,p)$ is a quasi-generic realisation of $G$.
Then $f_G(p)$ is generic.
\end{lemma}

We use $\ov{K}$ to denote the algebraic closure of a field $K$. Note
that $td[\ov{K}:K]=0$. We say that $G$ is {\em minimally rigid} if
it is rigid but $G-e$ is not rigid for any edge $e$; equivalently
$G$ is both rigid and independent. The following lemma
relates $\overline{\rat(p)}$ and $\overline{\rat(f_G(p))}$ when $G$ is minimally rigid.

\begin{lemma}\cite{JJ2}\label{qgen2}
Let $G$ be a minimally rigid mixed graph and
$(G,p)$ be a realisation of $G$ with no vertical direction edges and
with $p(v)=(0,0)$ for some vertex $v$ of $G$.
If $f_G(p)$ is generic
then $\overline{\rat(p)}=\overline{\rat(f_G(p))}$.
\end{lemma}

Lemmas \ref{qgen1} and \ref{qgen2} imply the following result for
rigid mixed graphs.

\begin{cor}\label{qgen3}
Let $G$ be a rigid mixed graph and
$(G,p)$ be a quasi-generic realisation of $G$
with $p(v)=(0,0)$ for some vertex $v$ of $G$.
Then $\overline{\rat({p})}=\overline{\rat(f_G({p}))}$.
\end{cor}
\bproof Let $H$ be a minimally rigid spanning subgraph of $G$. By Lemma \ref{qgen1},
$f_H(p)$ is generic. Hence Lemma \ref{qgen2} gives
$\overline{\rat({p})}=\overline{\rat(f_H({p}))}$. It is not difficult to see that
$\overline{\rat(f_H({p}))}\subseteq \overline{\rat(f_G({p}))}\subseteq \overline{\rat({p})}$.
Thus $\overline{\rat({p})}=\overline{\rat(f_G({p}))}$.
\eproof

We also need the following lemma, which implies that every
realisation of a rigid mixed graph which is equivalent to a generic
realisation is quasi-generic.

\begin{lemma}\cite{JJ2}\label{equiv}
Let $(G,p)$ be a quasi-generic realisation of a rigid mixed graph $G$.
Suppose that $(G,q)$ is equivalent to $(G,p)$
and that $p(v)=(0,0)=q(v)$ for some vertex $v$ of $G$.
Then $\overline{\rat(p)}=\overline{\rat(q)}$,
so $(G,q)$ is quasi-generic.
\end{lemma}

\section{Realisations of graphs with given direction constraints}\label{sec:realise}

Here we give a result concerning the realisation of a graph as a
direction-pure framework with given directions for its edges.
We need the
following concepts, introduced by Whiteley in \cite{W2}. A {\em
frame} is a graph $G=(V,E)$ together with a map $q:E\to \real^2$.
The {\em incidence matrix} of the frame $(G,q)$ is an $|E|\times
2|V|$ matrix $I(G,q)$ defined as follows. We first choose an
arbitrary reference orientation for the edges of $E$. Each edge in
$E$ corresponds to a row of $I(G,q)$ and each vertex of $V$ to two
consecutive columns. The submatrix of $I(G,q)$ with row labeled by
$e=uv\in E$ and pairs of columns labeled by $x\in V$ is $q(e)$ if $x=u$, is
$-q(e)$ if $x=v$, and is the $2$-dimensional zero vector otherwise.
It is known (see \cite{W2}) that when $q$ is generic, $I(G,q)$ is a
linear representation of $M_2(G)$ (the matroid union of two copies
of the cycle matroid of $G$). Thus we may use the characterisation
of independence in $M_2(G)$ given by Nash-Williams \cite{NW} to
determine when $I(G,q)$ has linearly independent rows. For
$X\subseteq V$, let $i_G(X)$ denote the number of edges of $G$
between vertices in $X$.

\begin{theorem}\label{frame} Suppose $G=(V,E)$ is a graph and $q:E\to \real^2$
is generic. Then the rows of $I(G,q)$ are linearly independent if
and only if $i_G(X)\leq 2|X|-2$ for all $\es\neq X\subseteq V$.
\end{theorem}

We can use this result to show that a  graph $G=(V,E)$ satisfying
$i_G(X)\leq 2|X|-3$ for all $X\subseteq V$ with $|X|\geq 2$ can be
realised as a direction-pure framework with a specified
algebraically independent set of slopes for its edges, and that this
realisation is unique up to translation and dilation when
$|E|=2|V|-3$.
Note that
given any realisation of $G$, we can always translate a specified
vertex $z_0$ to $(0,0)$ and dilate to arrange any specified distance
$t$ between a specified pair of vertices $x,y$.

\begin{theorem}\label{realise1}
Let $G=(V,E)$ be a graph such that $i_G(X)\leq 2|X|-3$ for all
$X\subseteq V$ with $|X|\geq 2$. Let $s$ be an injection from $E$ to
$\real$ such that $\{s_e\}_{e\in E}$ is generic.
Suppose $x_0,y_0,z_0\in V$ and $t\neq 0$ is a real number.
Then there exists an injection $p:V\to \real^2$ such that $\|p(x_0)-p(y_0)\|=t$,
$p(z_0)=(0,0)$ and, for all $e=uv\in E$, $p(u)-p(v) \in \langle(1,s_e)\rangle$.
Furthermore, if $|E|=2|V|-3$, then $p$ is unique up to dilation by $-1$ through $(0,0)$.
\end{theorem}

\bproof
We will construct $p$ as a combination of vectors in the nullspaces of certain frames.
First consider a generic frame $q$ on $G$ such that
$q(e)$ is a scalar multiple of $(-s_e,1)$ for every $e \in E$.
Then for any $p$ in the nullspace of $I(G,q)$
and $e=uv \in E$ we have $p(u)-p(v) \in \langle(1,s_e)\rangle$.
However, $p$ need not be injective. To address this issue,
we instead choose a pair of vertices $x,y\in V$, and consider
the graph $H$ obtained by adding the edge $f=xy$ to $G$
(which may be parallel to an existing edge).
Now let $(H,q)$ be a generic frame such that $q(e)$ is a scalar
multiple of $(-s_e,1)$ for every edge $e$ of $G$, and $q(f)$ is chosen arbitrarily (subject to the condition that $q$
should be generic).
For all $X\subseteq V$ with $|X|\geq 2$, we have
$i_H(X)\leq i_G(X)+1\leq 2|X|-2$ by hypothesis.
Theorem \ref{frame} now implies that the incidence matrix $I(H,q)$
has  linearly independent rows.
Thus $\rank I(H,q)=\rank I(G,q|_E)+1$.
Writing $Z_H$ for the null space of $I(H,q)$
and $Z_G$ for the null space of $I(G,q|_E)$,
we have $\dim Z_G = \dim Z_H+1$,
so we can choose $p_{f} \in Z_G \setminus Z_H$.
Then we necessarily have $p_f(x)\neq p_f(y)$.
Taking a suitable linear combination of the vectors $p_{f}$,
for all possible new edges $f=xy$, $x,y\in V$, we may
construct a vector $p$ in $Z_G$ with $p(x)\neq p(y)$ for all $x,y\in V$.
Since $p_f(u)-p_f(v) \in \langle(1,s_e)\rangle$ for each $f$
we also have $p(u)-p(v) \in \langle(1,s_e)\rangle$.
Furthermore, as noted before the proof, we can translate
and dilate to satisfy the other conditions, thus constructing the required map $p$.

We next show uniqueness when $|E|=2|V|-3$.
We have $\dim Z_G=2|V|-\rank I(G,q|_E)=2|V|-|E|=3$.
Define $p_1,p_2:V\to\real^2$ by $p_1(v)=(1,0)$ and $p_2(v)=(0,1)$ for all $v\in V$.
Note that $p_1,p_2\in Z_G$. Also, $p,p_1,p_2$ are linearly independent,
since $p(z_0)=(0,0)$, $p_1(z_0)=(1,0)$ and $p_2(z_0)=(0,1)$,
so $\{p,p_1,p_2\}$ is a basis for $Z_G$.
Now suppose that $p':V\to \real^2$ has the properties described
in the first part of the lemma. Then $p'\in Z_G$ so
$p'=ap+bp_1+cp_2$ for some $a,b,c\in \real$. Since
$p'(z_0)=p(z_0)=(0,0)$ we have $b=c=0$. Since
$\|p'(x_0)-p'(y_0)\|=t=\|p(x_0)-p(y_0)\|$ we have $p'\in \{p,-p\}$.
\eproof

The uniqueness part of this lemma, together with Lemma \ref{qgen1}, gives the following two results of
Whiteley, and Servatius and Whiteley.

\begin{lemma}\label{direction}\cite{W}
Suppose that $(G,p)$ is a generic direction-pure framework.
Then $(G,p)$ is direction globally rigid if and only if it is direction-rigid.
\end{lemma}

\begin{lemma}\label{special}\cite{SW}
Suppose that $(G,p)$ is a generic realisation of a mixed graph $G=(V;D,L)$.
If $G$ is rigid and $|L|=1$ then $(G,p)$ is globally rigid.
\end{lemma}

\section{Direction reduction}\label{sec:dirred}

In this section we prove Theorem \ref{reduce}. We also prove a lemma
which determines when a rigid direction-independent mixed graph is
direction reducible. For the proof of Theorem \ref{reduce} we deal
with the two reduction operations separately.

\begin{lemma}\label{Dind}
Suppose $(G,p)$ is a generic realisation of a mixed graph
$G=(V;D,L)$ and that $e=uv\in D$ belongs to a direction-pure
$M$-circuit $H=(U;F,\es)$ of $G$.
Then $(G,p)$ is globally rigid if and only if
$(G-e,p)$ is globally rigid.
\end{lemma}

\bproof
If $(G-e,p)$ is globally rigid then $(G,p)$ is clearly globally rigid.
Conversely, suppose that $(G,p)$ is globally rigid and $(G-e,q)$ is equivalent to $(G-e,p)$.
Since $H$ is a direction-pure circuit,
both $(H,p|_U)$ and $(H-e,p|_U)$ are direction-rigid.
Hence $(H-e,p|_{U})$ is globally direction-rigid by Lemma \ref{direction}.
Thus $q(u)-q(v)$ is a scalar multiple of $p(u)-p(v)$, and hence $(G,q)$ is equivalent to $(G,p)$.
Since $G$ is globally rigid, $q$ is congruent to $p$.
This shows that $(G-e,p)$ is globally rigid.
\eproof

\begin{lemma}\label{2tree}
Let $(G,p)$ be a quasi-generic realisation of a rigid
mixed graph $G=(V;D,L)$. Suppose that $G$ has a proper induced
subgraph $H=(U;F,L)$ such that
the graph $G/H$ obtained by contracting $H$ to a single vertex (deleting all edges
contained in $H$ and keeping all other edges, possibly as parallel edges) has only direction edges and is
the union of two edge-disjoint spanning trees. Then  $(G,p)$ is
globally rigid if and only if $(H,p|_H)$ is globally rigid.
\end{lemma}

\bproof First suppose that $(H,p|_H)$ is globally rigid. Let $G'$
be constructed from $G$ by substituting $H$ by a minimally rigid graph $H'$
with exactly one length edge. Then $G'$ is rigid by Lemma
\ref{subrigid}. Since $G'$
is rigid and has exactly one
length edge, $(G',p)$
is globally rigid by
Lemma \ref{special}. Thus $(G,p)$ is globally rigid by Lemma
\ref{replace}.

Conversely, suppose that $(H,p|_H)$  is not globally rigid. Then there
exists an equivalent but non-congruent framework $(H,\tilde q)$.
Without loss of generality we may suppose that
$p(u)=(0,0)=\tilde q(u)$ for some $u\in V(H)$.
Let $D^*= D\sm F$ be the set of edges of $G/H$ and $m$ be the number of
vertices of $G/H$. Then $|D^*|=2m-2$, as $G/H$ is the union of two edge-disjoint
spanning trees. Since $G$ is rigid we have
$$2|V|-2=r(G)\leq |D^*|+r(H)\leq 2m-2+2|V(H)|-2=2|V|-2.$$
Thus equality must hold throughout. In particular, $r(H)=2|V(H))|-2$,
so $H$ is rigid.

We again consider the rigid mixed graph $G'=(V,D',L')$ with exactly one length edge defined in the first paragraph of the proof.
Since $G'$
has $|D^*|+2|V(H)|-2=2|V|-2$ edges, it is minimally rigid.
We will construct a framework $(G,q)$ which is equivalent
to $(G,p)$  and has $q|_H=\tilde q$ by applying Theorem \ref{realise1} to $G'$.

Define $s:D'\cup L'\to \real$ by $s(e)=s_{\tilde q}(e)$ for
$e\in F'$, $s(e)=l_{\tilde q}(e)$ for $e\in L'$, and $s(e)=s_p(e)$
for $e\in D^*$. We will use Theorem \ref{realise1} to construct a framework $(G',q)$ such that
$s_q(e)=s(e)$  for all edges $e$ of $G'$.
To do this, we
first need to show that $s|_{D'}$ is generic.
We will prove the stronger result that  $s$ is generic by showing
that $td[\ov{\mb{Q}(s)}:\mb{Q}]=|D'|+|L'|=2|V|-2$. We have
$td[\ov{\mb{Q}(p)}:\mb{Q}]=2|V|-2$, as $p$ is quasi-generic and $p(u)=(0,0)$, so it
suffices to prove that $\ov{\mb{Q}(s)}=\ov{\mb{Q}(p)}$. Since $G$ is
rigid, Corollary \ref{qgen3} gives
$\ov{\mb{Q}(f_G(p))}=\ov{\mb{Q}(p)}$. Also, $s$ is obtained from
$f_G(p)$ by replacing the values $f_H(p|_U)$ by the values
$f_{H'}(\tilde q)$, so we need to show that these generate the same
algebraic closure over $\mb{Q}$. Since $(H,\tilde q)$ is equivalent
to $(H,p|_U)$, Lemma \ref{equiv} gives $\ov{\mb{Q}(\tilde
q)}=\ov{\mb{Q}(p|_U)}$. Since $p|_U$ is quasi-generic, it follows
that $\tilde q$ is quasi-generic. Then, since $H$ and $H'$ are
rigid, two applications of Corollary \ref{qgen3}, give
$\ov{\mb{Q}(f_H(p|_U))}=\ov{\mb{Q}(p|_U)}$ and
$\ov{\mb{Q}(f_{H'}(\tilde q))}=\ov{\mb{Q}(\tilde q)}$. Putting these
three equalities together gives
\[\ov{\mb{Q}(f_{H'}(\tilde q))}=\ov{\mb{Q}(\tilde q)}
=\ov{\mb{Q}(p|_U)}=\ov{\mb{Q}(f_H(p|_U))},\] which is what we needed
to prove $\ov{\mb{Q}(s)}=\ov{\mb{Q}(p)}$. Therefore $s$ is generic.
Now we can apply  Theorem \ref{realise1}, with $x_0y_0$ equal to the
unique length edge of $G'$, to obtain a realisation $(G',q)$ with
$f_{G'}(q)=s$. By construction $(H',q|_{U})$ is equivalent to
$(H',\tilde q)$. But $H'$ is globally rigid by Lemma \ref{special},
so $q|_{U}$ is congruent to $\tilde q$. Hence we can apply a
translation, and possibly a dilation by $-1$, to obtain $q|_U=\tilde
q$.

Since $(H,\tilde q)$ is equivalent  to $(H,p|_H)$ and $s_q(e)=s(e)=s_p(e)$ for all $e\in D^*$, $(G,q)$
is equivalent to $(G,p)$ and satisfies $q|_U=\tilde q$.  Since $(H,\tilde q)$ is not congruent to
$(H,p|_U)$, $(G,q)$ is not congruent to $(G,p)$. Thus  $(G,p)$ is not globally rigid. \eproof

Theorem \ref{reduce} follows immediately from Lemmas \ref{Dind} and \ref{2tree}.
We close this section with a result which determines when a rigid, direction-independent mixed graph $G$ is direction reducible to a given subgraph $H$.
This lemma will be used in Section \ref{sec:con} to give an algorithm for finding a direction reduction.

\begin{lemma}\label{dindred}
Suppose $G=(V;D,L)$ is a rigid, direction-independent mixed graph and
$H=(V';D',L')$ is an induced proper subgraph of $G$.
Then $G$ is direction reducible to $H$
if and only if $L'=L$ and $|D\sm D'|=2|V\sm V'|$.
\end{lemma}
\bproof
Since $G$ is direction-independent, $G$ is direction reducible to $H$
if and only if $L'=L$ and the graph $F=(V'';D'',\es)$ obtained
by contracting $H$ to a single vertex $v_H$ is the union of two edge-disjoint spanning trees.
Thus, if $G$ is direction reducible to $H$, then we have
$L'=L$ and $|D\sm D'|=|D''|=2|V''|-2=2|V\sm V'|$.

We next assume that $L'=L$ and $|D\sm D'|=2|V\sm V'|$ and show that
$F$ is the union of two edge-disjoint spanning trees. Suppose not.
Then by a theorem of Nash-Williams \cite{NW}, there exists
$X\subseteq V''$ with $|X|\geq 2$ and $i_F(X)\geq 2|X|-1$. If
$v_H\not\in X$ then the fact that $G$ is direction-independent
implies that $i_F(X)=i_G(X)\leq 2|X|-3$. Thus $v_H\in X$. Since
$|D''|=|D\sm D'|=2|V''|-2$ there are at most
$(2|V''|-2)-(2|X|-1)=2|V''\sm X|-1$ edges in $F$ which are not
induced by $X$. It follows that
$$r(G)\leq r(G[V\sm (V''\sm X)])+(2|V''\sm X|-1)\leq (2|V\sm (V''\sm X)|-2)+(2|V''\sm X|-1)=2|V|-3.$$
This contradicts the hypothesis that $G$ is rigid. Thus $F$ is the union of two edge-disjoint spanning trees
and $G$ is direction reducible to $H$.
\eproof

\section{Direction irreducible mixed graphs}\label{sec:irreducible}

Theorem \ref{reduce} enables us to reduce the problem of
characterising globally rigid generic direction-length frameworks to
the case when the underlying graph is direction irreducible.
In this section we prove a structural lemma for direction irreducible mixed graphs
which have a globally rigid generic realisation even though they are not redundantly rigid. (This will be used in the next section to construct
two equivalent but non-congruent generic realisations of a mixed graph which is direction irreducible but not redundantly rigid.)
\begin{lemma}\label{dired}
Let $G=(V;D,L)$ be a direction
irreducible mixed graph which has $|L|\geq 2$ and is   not redundantly rigid. Suppose that $(G,p)$ is a  globally rigid generic realisation of $G$.
Then\\
(a) $G-e$ is bounded for all $e\in D$,\\
(b) $r(G-e)=r(G)-1$ for all $e\in D$, and  \\
(c) every length edge of $G$ belongs to a length-pure $M$-circuit of
$G$.
\end{lemma}
\bproof (a) First note that $G$ is direction-independent, since $G$
is direction irreducible. Now suppose for a contradiction that $G-e$
is not bounded for some $e\in D$. We will show that $G$ has a
direction reduction. Let $H_1,H_2,\ldots,H_m$ be the bounded
components of $G-e$. Then each length edge of $G$ is contained in one of the subgraphs $H_i$.
Let $D^*\sub D$ be the set of all edges of
$G$ joining distinct subgraphs $H_i$, and $H$ be the graph obtained from $G$ by contracting each $H_i$ to a single vertex. Since $G$ is rigid, $G$ is
bounded. Since $G$ is direction-independent, Lemma
\ref{bounded} now implies that the graph $G/L$ obtained from $G$ by contracting
each length edge has two edge-disjoint spanning trees.  Since $H$ can be obtained from $G/L$ by contracting a (possibly empty) set of direction edges, $H$ also has two edge-disjoint spanning trees. In particular,  $|D^*|\geq 2m-2$. On the other hand, Lemma
\ref{boundedcomps} implies that $|D^*-e|\leq 2m-3$. Thus $e\in D^*$,
$|D^*|= 2m-2$, and $H$  is the union of two edge-disjoint
spanning trees.
 Since $G$ is rigid we have
$$2|V|-2=r(G)\leq |D^*|+\sum_{i=1}^m r(H_i)\leq 2m-2+\sum_{i=1}^m (2|V(H_i)|-2)=2|V|-2.$$
Thus equality must hold throughout. In particular,
$r(H_i)=2|V(H_i)|-2$ for each $i$, so each subgraph $H_i$ is
rigid.

Let $G'=(V;D',L')$ be obtained from $G$ by substituting each
non-trivial subgraph $H_i$ by a minimally rigid graph
$H'_i$ with exactly one length edge. Each framework $(H'_i,p|_{H_i'})$ is globally
rigid by Lemma \ref{special}. Thus repeated applications of Lemma
\ref{replace} imply that $(G',p)$ is globally rigid. On the other
hand, $|D'|+|L'|=|D^*|+\sum_{i=1}^m
r(H_i)=2m-2+\sum_{i=1}^m(2|V(H_i)|-2)=2|V|-2$, so $G'$ is minimally
rigid. Theorem \ref{hend}(a) now implies that $G'$ has exactly one
length edge. Since $H_i'$ contains a length edge whenever $H_i$ is
non-trivial, $G-e$ has exactly one non-trivial bounded component,
$H_1$ say. Since $G/H_1=H$ and $H$ is the union of two edge-disjoint spanning trees, $G$ is direction reducible to
$H_1$. This contradicts the hypothesis that $G$ is direction
irreducible.
\\[1mm]
(b) Suppose that $r(G-e)=r(G)$ for some $e\in D$. Then $e$ is
contained in an $M$-circuit $H$ of $G$. Since $G$ is direction-independent,
$H$ must be a mixed $M$-circuit.
Since $G$ is not redundantly rigid, $G-f$ is not rigid for some
$f\in D\cup L$. Theorem \ref{hend}(a) implies that $f\in D$. Clearly
$f$ is not an edge of $H$ and hence $H$ is a non-trivial rigid
subgraph of $G-e$. Theorem \ref{hend}(b) now implies that $G-f$ is
unbounded, contradicting (a).
\\[1mm]
(c) Choose $e\in L$. Then $e$ belongs to an $M$-circuit $H$ of $G$
by Theorem \ref{hend}(a). By (b), $H$ cannot be a mixed $M$-circuit.
Hence $H$ is length-pure.
\eproof

\section{Proof of Theorem \ref{redrigid}}
\label{sec:mainresult}

Since every generic realisation of a direction irreducible, 2-connected, direction-balanced, redundantly rigid graph is globally rigid by  Theorem \ref{thm:katie} and Lemma \ref{lem:redMcon}, we only need to show necessity in Theorem \ref{redrigid}. Hence we may suppose that $G$ is a direction irreducible mixed graph and that every generic realisation of $G$ is globally rigid. Then $G$ is 2-connected and direction-balanced by \cite{JJ1}. We will complete the proof by applying Theorem \ref{thm:ExistsNonGloballyRigidRealisation} below to deduce that $G$ must also be redundantly rigid. The proof idea is to show that, if $G$ is not redundantly rigid, then for any given generic realisation $(G,p)$,
we can construct a sequence of generic realisations $q_0, q_1, \ldots, q_t$ such that $t\leq |D|$ and $(G,q_t)$ is not globally rigid.
We construct this sequence from $(G,p)$ by first reflecting $(G,p)$ in the $x$-axis to obtain $(G,q_0)$, and then recursively ``correcting'' the changed direction constraints back to their original value in $(G,p)$. Every time we ``correct'' a direction constraint, we obtain a new realisation in our sequence.



\begin{theorem}\label{thm:ExistsNonGloballyRigidRealisation}
Let $G=(V;D,L)$ be a direction irreducible mixed graph with $|{L}|\geq 2$ such that $G$ is not redundantly rigid. Then some generic realisation of $G$  is not globally rigid.
\end{theorem}

\bproof
We shall proceed by contradiction. Assume that  all generic realisations of $G$ are globally rigid.
By Lemma \ref{dired}(b) and (c), every length edge of $G$ is contained in a length-pure circuit in the rigidity matroid of $G$, and no direction edge of $G$ is contained in any circuit.
Let $D=\{d_0,d_1,\ldots,d_k\}$, let $G_1=(V_1;\emptyset,L_1)$ be a non-trivial $M$-connected component of $G$ and let $v_0\in V_1$.

Let $(G,p)$ be a quasi-generic realisation of $G$ with $p(v_0)=(0,0)$ and let $(G,q_0)$ be the quasi-generic realisation obtained by
reflecting $(G,p)$ in the $x$-axis. Then $(G-D,p)$ is equivalent to $(G-D,q_0)$. In addition we have  $s_{q_0}(d_i)= -s_p(d_i)$ for all $d_i\in D$, so $(G,p)$ and $(G,q_0)$ are not equivalent.


\begin{claim}\label{clm:MainResult_RealisationsInSequenceAreQuasiGeneric}
For all $j\in\{ 0,1,\ldots, k\}$ there exists a quasi-generic framework $(G,q_j)$  with $q_j(v_0)=(0,0)$,  rigidity map $f_G(q_j)=(h_{q_j}(e))_{e\in E}$ given by
\[
h_{q_j}(e)=
\begin{cases}
    s_{q_0}(e)  & \text{when } e\in\{d_{j},d_{j+1},\ldots,d_k\}\\
    h_p(e)          & \text{otherwise,}
\end{cases}
\]
and with the property that that $(G_1,q_j|_{V_1})$ can be obtained from $(G_1,q_0|_{V_1})$ by a rotation about the origin.
\end{claim}

\bproof
We proceed by induction on $j$. If $j=0$ then  the claim holds trivially for $(G,q_0)$. Hence suppose that the required framework $(G,q_j)$ exists for some
$0\leq j<k$.
        The quasi-generic framework  $(G-d_j,q_j)$ is bounded but not rigid by Lemma \ref{dired}(a) and (b) (since boundedness and rigidity are generic properties). Since $(G,q_j)$ is globally rigid by assumption, Lemma \ref{deletedirection} implies that we can continuously move $(G-d_j,q_j)$ to form $(G-d_j,-q_j)$ whilst keeping $v_0$ fixed at the origin and maintaining all edge constraints. During this motion, the direction of the missing edge $d_{j+1}=u_{j+1}v_{j+1}$ changes continuously from $q_j(v_{j+1})-q_j(u_{j+1})$ to $-(q_j(v_{j+1})-q_j(u_{j+1}))$, a rotation by $180^\circ$. So at some point in this motion we must pass through a realisation $(G-d_{j+1},q_{j+1})$ at which the slope of this missing edge is $s_p(d_{j+1}))$. We can now add the edge $d_j$ back to this realisation to obtain the desired framework $(G,q_{j+1})$. Note that since $G_1$ is a length rigid subgraph of $G-d_j$ and the motion of $(G-d_{j+1},q_j)$ is continuous and keeps $v_0$ fixed at the origin, $(G_1,q_{j+1}|_{V_1})$ can be obtained from $(G_1,q_j|_{V_1})$ by a rotation about the origin.

        It remains to show that $(G,q_{j+1})$ is quasi-generic. Let $H$ be a minimally rigid spanning subgraph of $G$. Since $h_{q_{j+1}}(e)=\pm h_{p}(e)$ for all $e\in E(G)$ we have $\rat(f_H(q_{j+1}))=\rat(f_H(p))$. Since $f_H(p)$ is generic by Lemma \ref{qgen1}, Lemma \ref{qgen2} implies that
        $$\tran [\overline{\rat(q_{j+1})},\rat]=\tran [\overline{\rat(f_H(q_{j+1}))},\rat]=\tran [\overline{\rat(f_H(p))},\rat]=2|V|-2.$$ We can now use Lemma \ref{quasi} to deduce that $(H,q_{j+1})$, and hence also $(G,q_{j+1})$, are quasi-generic.
\eproof

    Applying  Claim \ref{clm:MainResult_RealisationsInSequenceAreQuasiGeneric} with $j=k$, we obtain a quasi-generic realisation $q_k$ of $G$ which is equivalent to $(G,p)$, has $q_k(v_0)=(0,0)$, and is such that $(G_1,q_{k}|_{V_1})$ can be obtained from $(G_1,q_0|_{V_1})$ by a rotation about the origin. Since $q_0$ was obtained from $p$ by reflecting $V_1$ across the $x$-axis, we have
\[
q_k(v) = R Z  p(v) \quad\text{for all } v\in V_1
\]
where $R$ and $Z$ are the $2\times 2$ matrices representing this rotation and reflection. Since $(G_1,p|_{V_1})$ is a quasi-generic framework with at least four vertices and $RZ$ acts on $\real^2$ as a reflection in some line through the origin, we have $q_k(v)\neq \pm p(v)$ for some $v\in V_1$.
 Hence $q_k|_{V_1}$ is not congruent to $p|_{V_1}$, and $q_k$ is not congruent to $p$. This implies that $(G,p)$ is not globally rigid and contradicts our initial assumption that all generic realisations of $G$ are globally rigid.
\eproof

\section{Algorithmic considerations}\label{sec:con}

We will describe a polynomial algorithm which decides if every generic realisation of a given mixed graph $ G=(V;D,L) $ is globally rigid.
If $|L|\leq 1$ then we need only determine whether $G$ is rigid and this can be accomplished using an orientation algorithm as in \cite {BJ} or a pebble game algorithm as in \cite{LS}. Hence we may suppose that $|L|\geq 2$.

We first consider the case when $ G $ is direction irreducible. In this case Theorem \ref{redrigid} tells us we need only determine whether $ G $ is 2-connected, direction-balanced and redundantly rigid. The first two properties can be checked using the connectivity algorithm of \cite{HT}, and the third by an orientation or pebble game algorithm.

It remains to show how we can reduce $ G $ to the direction irreducible case when $ G $ is direction reducible. We do this in two stages. In the first stage we reduce $ G$ to a direction-independent graph $G'=(V; D', L) $ by choosing $ D'$ to be a maximal subset of $ D $ which is independent in $ M (G) $.  This may again be accomplished using an orientation or pebble game algorithm.

Our final step is to find a direction reduction for a direction-independent graph. We accomplish this by using the following lemma combined with the algorithm for determining the bounded components in a mixed graph given in \cite{JK1}.

\begin{lemma} Suppose $G=(V; D, L)$ is a rigid, direction-independent mixed graph. Then  $G$ is direction reducible if and only if  $G-e$ is unbounded and has exactly one nontrivial bounded component for some  $e\in D$. Furthermore, if $G-e$ is unbounded and has exactly one nontrivial bounded component  $H$ for some  $e\in D$, then  $G$ is direction reducible to $H$.
\end{lemma}
\bproof We first suppose that $G$ is direction reducible to a subgraph $G'=(V'; D', L)$ . By Lemma \ref{dindred}, $|D\sm D'|=2|V\sm V'|$. It follows that, for any  $ e\in D\sm D'$, the graph obtained from $ G-e $ by contracting $ E (G') $ has $2|V\sm V'|-1$ edges and $|V\sm V'|+1$ vertices, so does not have two edge-disjoint spanning trees. This implies that $(G-e)/L $ does not have two edge-disjoint spanning trees so $ G-e $ is unbounded by Lemma \ref{bounded}. In addition,  we have
$$2|V|-2=r (G)\leq r (G') +|D\sm D'|\leq 2|V'|-2+2|V\sm V'|$$
and equality must hold throughout. In particular,  $ r (G')=2|V'|-2$, so $ G'$ is rigid, and hence bounded. Since $ L\subseteq E (G') $, $ G'$ must be the unique nontrivial bounded component of $G-e$.

We next suppose that $G-e$ is unbounded and has exactly one nontrivial bounded component  $ H=(V'; D', L) $ for some  $e\in D$. Then
$|(D-e)\sm D'|\leq 2(|V\sm V'|+1)-3 $ by Lemma \ref{boundedcomps}, so $|D\sm D'|\leq 2|V\sm V'|$. Strict inequality cannot hold since $ G $ is rigid, and hence bounded. Thus $|D\sm D'|= 2|V\sm V'|$ and $G$ is direction reducible to $ H $ by Lemma \ref{dindred}.
\eproof

\section{Closing remarks}\label{sec:close}

The question of deciding whether global rigidity  is a generic property of direction-length frameworks remains open.
Theorem \ref{thm:katie} shows that it is a generic property when the underlying graph is $M$-connected,
 and the necessary conditions for global rigidity given in \cite{JJ1} show that it is also a generic property if the underlying graph is not both 2-connected and direction-balanced.
Theorem \ref{reduce} and Lemma \ref{lem:redMcon} reduce the question to the case when the underlying graph is direction irreducible and is not redundantly rigid. Theorem \ref{redrigid} tells us that a direction irreducible mixed graph  $G$ which is not redundantly rigid has a generic realisation which is not globally rigid, but it is still conceivable that $G$ may also have a generic realisation which is globally rigid. We believe that this is not the case:

\begin{conj}\label{characterise}
Suppose $(G,p)$ is a generic realisation of a direction irreducible
mixed graph $G$ with at least two length edges. Then $(G,p)$ is
globally rigid if and only if $G$ is $2$-connected,
direction-balanced, and redundantly rigid.
\end{conj}

Conjecture \ref{characterise} (combined with Theorem \ref{reduce} and Lemma \ref{special}) would give a complete characterization of globally rigid generic direction-length frameworks.
By the above discussion, the conjecture would follow if we could show that no generic realisation of a direction irreducible, non-redundantly rigid mixed graph with  at least two length edges is globally rigid.
We can try to do this by using the proof technique
 of Theorem \ref{hend}(b) given  in
\cite{JK2}. The idea is to remove a `non-redundant' direction edge
$e$ from a generic rigid framework $(G,p)$ and allow $(G-e,p)$ to move, keeping one vertex pinned at the origin. We
eventually reach another realisation $(G,q)$ for which $e$ has the same
direction as in $(G,p)$. Formally, we show that the connected component $C$
of the framework space $S_{G-e,p,v}$ containing $p$ must also
contain a point $q\neq p$ such that $(G,p)$ is equivalent to $(G,q)$.
The problem is that we may have $q=-p$.
A simple example where this occurs is the
case when $|D|=1$, when it is easy to see that the only points $q\in
C$ with $(G,q)$ equivalent to $(G,p)$ are $q=p$ and $q=-p$.
Our next result gives a more interesting family of graphs $G$ for which $-p\in C$.

\begin{lemma}\label{lem:-p}
Let $G=(V;D,L)$ be a rigid mixed graph, $H=(U;\emptyset,L)$ be the length-pure subgraph induced by $L$ and $u\in U$.
Suppose that $H$ is length rigid, $r(G-e)=r(G)-1$ for all $e\in D$, and $G-e_0$ is bounded for some $e_0\in D$.
Let $(G,p)$ be a quasi-generic framework with $p(u)=(0,0)$ and $C$ be the connected component of the configuration space
$S_{G-e_0,p,u}$ which contains $p$. Then $-p\in C$.
\end{lemma}
\bproof The idea is to rotate $(H,p|_U)$ by $\theta$ radians about $p(u)=(0,0)$ and use Theorem \ref{realise1} to show that, for almost all values of $\theta$, we can extend the resulting framework $(H,q_\theta)$ to a framework $(G-e_0,p_\theta)$ which is equivalent to  $(G-e_0,p)$. To apply Theorem  \ref{realise1}, we construct $G'$ from $G$ by substituting a minimally rigid graph $H'$ with exactly one length edge for $H$ and then show that the required set of edge slopes for $(G'-e_0,p_\theta)$ is algebraically independent over $\rat$.

Let $H'=(U;D',L')$ be a minimally rigid graph on the same vertex set
as $H$ with exactly one length edge and let $G'$ be obtained
from $G$ by replacing $H$ by $H'$.  We first show that $G'-e_0$ is minimally rigid.
Since $G$ is rigid, $H$ is length-rigid and $r(G-e)=r(G)-1$ for all $e\in D$, we have
$|D|=2|V|-2-(2|U|-3)$ and hence $|D-e_0|=2|V|-2|U|$. Since $H'$ has $2|U|-2$ edges, this implies that $G'$ has $2|V|-2$ edges.  It remains to show that $G'-e_0$ is rigid.
Since $G-e_0$ is bounded, $(G-e_0)^+$ is rigid by Lemma \ref{augment}. Since $G'-e_0$ can be obtained from $(G-e_0)^+$ by substituting $H^+$ with $H'$, it is rigid by Lemma \ref{subrigid}.
Therefore $G'-e_0$ is minimally rigid.

For each $\theta\in [0,2\pi)$ let $q_\theta:U\to \real^2$ be the configuration
obtained by an anticlockwise rotation of $p|_U$ through $\theta$ radians about $(0,0)$.
Write $B=\{q_\theta \;:\;\theta\in [0,2\pi)\}$,
and let $B^*$ be the set of all configurations $q_\theta\in B$ such that
the set of slopes $\{s_p(e)\}_{e\in D-e_0}\cup \{s_{q_\theta}(e)\}_{e\in D'}$
is defined and is algebraically independent over $\rat$.
We claim that $B^*$ is a dense subset of $B$.
First we note that $q_0=p|_U \in B^*$, as $G'-e_0$ is independent,
so Lemma \ref{qgen1} implies that $f_{G'}(p)$ is generic.
To see the effect of a rotation by $\theta$, consider an edge $e=v_1v_2$ in $D'$
and let $(x_1,y_1)$ and $(x_2,y_2)$ be the co-ordinates of $v_1$ and $v_2$ in $p$.
Co-ordinates in $q_\theta$ are obtained by applying the transformation
$R_\theta = \left( \begin{array}{cc} \cos \theta & -\sin \theta \\ \sin \theta & \cos \theta \end{array} \right)$,
so we have
\[ s_{q_0}(e) = s_p(e) = \frac{y_1-y_2}{x_1-x_2} \mbox{ and }
  s_{q_\theta}(e) = \frac{(x_1-x_2)\sin\theta + (y_1-y_2)\cos\theta}{
  (x_1-x_2)\cos\theta - (y_1-y_2)\sin\theta}, \mbox{ so}\]
\[ s_{q_\theta}(e) = r(s_p(e),\tan\theta), \mbox{ where } r(s,t) = \frac{t+s}{1-st}.\]

Consider any non-zero polynomial $z$ with rational coefficients and $|D-e_0|+|D'|$ variables,
labelled as ${\bold s}=(s_e:e\in D-e_0)$ and ${\bold s}'=(s'_e:e \in D')$.
Substituting ${\bold s}=(s_p(e):e\in D-e_0)$ and ${\bold s}'=(s_{q_\theta}(e):e \in D')$
into $z$ gives a rational function $z^*$ in $(s_p(e):e\in (D-e_0) \cup D')$ and $\tan\theta$.
Note that $z^*$ is not identically zero, as it is non-zero when $\theta=0$ by the hypothesis that $p$ is quasi-generic.
Thus there are only a finite number of values of $\theta \in [0,2\pi)$ for which $z^*$ is zero.
Furthermore, the number of such polynomials $z$ is countable, so there are
only countably many $\theta$ for which $\{s_p(e)\}_{e\in D-e_0}\cup \{s_{q_\theta}(e)\}_{e\in D'}$
is algebraically dependent over $\rat$. Thus $B \sm B^*$ is countable,
so in particular $B^*$ is a dense subset of $B$.

For each $q_\theta\in B^*$, we can apply Lemma \ref{realise1} to obtain a configuration
$p_\theta:V\to \real^2$ such that $l_{p_\theta}(e_1)=l_{p}(e_1)$,  where $e_1$ is the unique length edge of $G'$,
$p_\theta(u)=(0,0)$, $s_{p_\theta}(e)=s_p(e)$ for $e\in D-e_0$ and $s_{p_\theta}(e)=s_{q_\theta}(e)$ for ${e\in D'}$.
Since $(H',q_\theta)$ is globally rigid we have $p_\theta|_U\in \{q_\theta,-q_\theta\}$.
Hence $(G-e_0,p_\theta)$ is equivalent to $(G-e_0,p)$. Replacing $p_\theta$ by $-p_\theta$ if
necessary, we may suppose that $p_\theta|_U=q_\theta$; this determines $p_\theta$ uniquely
by Lemma \ref{realise1}. Now note that the defining conditions of $p_\theta$ are polynomial
equations with coefficients that are continuous functions of $\theta$, except at a finite
set of exceptional values for $\theta$ corresponding to vertical edges in $p_\theta$.
Since $B^*$ is a dense subset of $B$, it follows that $\{p_\theta:q_\theta \in B^*\}$
all belong to the same component of the framework space $S_{G-e_0,p,u}$,
which is $C$, since $q_0=p|_U \in B^*$. Now note that $q_\pi \in B^*$,
as $s_{q_\pi}(e)=-s_p(e)$ for $e \in D'$, so $\{s_p(e)\}_{e\in D-e_0}\cup \{s_{q_\pi}(e)\}_{e\in D'}$
generates the same extension of $\rat$ as  $\{s_p(e)\}_{e\in D-e_0}\cup \{s_p(e)\}_{e\in D'}$.
Therefore $p_\pi \in C$. Since $p_\pi = -p$ by the uniqueness property noted above, $-p \in C$.
\eproof

Lemma \ref{lem:-p} suggests that the family of graphs satisfying the hypotheses of this lemma may be a source of counterexamples to Conjecture \ref{characterise}. Our next result shows that this is not the case.

\begin{theorem}\label{vspecial}
Let $(G,p)$ be a generic realisation of a rigid graph $G=(V;D,L)$. Suppose
that $L$ induces
a length-rigid subgraph of $G$ with at least two edges and $r(G-e)=r(G)-1$ for all $e\in D$. Then $(G,p)$ is not globally rigid.
\end{theorem}

\bproof We proceed by contradiction. Suppose the theorem is false
and choose a counterexample $(G,p)$ such that $G$ is as small as
possible. If $G$ were direction reducible to a subgraph $F$ then we
could apply induction to deduce that $(F,p|_{F})$ is not globally rigid.
Then $(G,p)$ is not globally rigid by Theorem \ref{reduce},
which is a contradiction. Hence $G$ is direction
irreducible. Then, by Lemma \ref{dired},
$G-e$ is bounded for all $e\in D$.

Let $H=(U;\es,L)$ be the length-rigid subgraph of $G$ induced by $L$. Choose
$u\in U$ and $e_0\in D$. By translation we can
replace the assumption that $(G,p)$ is generic by the assumption
that $(G,p)$ is quasi-generic and $p(u)=(0,0)$.
Let $H'=(U;D',L')$ be a minimally rigid graph on the same vertex set
as $H$ with exactly one length edge, $f$, and let $G'$ be obtained
from $G$ by substituting $H$ by $H'$. We can show that $G'$ is minimally rigid as in the proof of Lemma
\ref{lem:-p}.

Let $(H',q)$ be obtained from $(H',p|_U)$ by reflection in the $x$-axis.
Then $s_q(e)=-s_p(e)$ for all $e \in D'$.
Since $\{s_p(e)\}_{e\in D-e_0}\cup \{s_p(e)\}_{e\in D'}$ is generic,
$\{s_p(e)\}_{e\in D-e_0}\cup \{s_q(e)\}_{e\in D'}$ is generic.
Thus we can apply Lemma \ref{realise1}  to obtain
$p':V\to \real^2$ such that $l_{p'}(f)=l_{p}(f)$,
$p'(v)=(0,0)$, $s_{p'}(e)=s_p(e)$ for $e\in D-e_0$ and
$s_{p'}(e)=s_q(e)$ for $e\in D'$.
We have $\rat(f_{G'-e_0}(p'))=\rat(f_{G'-e_0}(p))$,
so $p'$ is quasi-generic by Lemma \ref{qgen2}.
Now consider $(G-e_0,p')$ and let
$C$ be the connected component of the framework space
$S_{G-e_0,p',u}$ which contains $p'$.
By Lemma
\ref{lem:-p}, we have have $-p' \in C$.

The remainder of the proof is similar to that of \cite[Theorem 1.3]{JK2}. Let $e_0=u_0v_0$. For any $p'' \in C$ let
$F(p'')=(p''(u_0)-p''(v_0))/\|p''(u_0)-p''(v_0)\|$
be the unit vector in the direction of $p''(u_0)-p''(v_0)$;
this is well-defined since we never have $p''(u_0)=p''(v_0)$ by \cite[Lemma
3.4]{JJ2}. Consider a
path $P$ in $C$ from $p'$ to $-p'$. Then $F(p'')$ changes continuously
from $F(p')$ to $-F(p')$ along $P$. By the intermediate value theorem there
must be some $p'' \in P$ such that $F(p'')$ is either $F(p)$ or
$-F(p)$. Then $(G,p'')$ is equivalent to $(G,p)$. On the other hand
$p''$ is not congruent to $p$ since $p''|_U$ is obtained from $p|_U$
by a reflection (as well as a translation and a rotation). It
follows that $(G,p)$ is not globally rigid. \eproof

Lemma \ref{dired} and Theorem \ref{vspecial} imply that Conjecture
\ref{characterise} holds for mixed graphs whose length edges induce
a length rigid subgraph. We next verify Conjecture
\ref{characterise} for mixed graphs $G$ with at most $2|V|-1$ edges.

\begin{theorem}\label{sparse}
Suppose $(G,p)$ is a generic realisation of a direction irreducible
mixed graph $G=(V;D,L)$ with $|L|\geq 2$ and $|D|+|L|\leq 2|V|-1$. Then
$(G,p)$ is globally rigid if and only if $G$ is $2$-connected,
direction-balanced, and redundantly rigid.
\end{theorem}

\bproof
As noted at the beginning of this section, we only need to show that if $(G,p)$ is a globally rigid generic realisation of a direction irreducible
mixed graph $G=(V;D,L)$ with $|L|\geq 2$ and $|D|+|L|\leq 2|V|-1$, then
$G$ is redundantly rigid.
Suppose not. Then
Theorem \ref{dired} implies that $r(G-e)=r(G)-1$ for all $e\in D$
and every $f\in L$ is contained in a length-pure $M$-circuit of $G$.
The facts that $r(G)=2|V|-2$ and $|D|+|L|\leq  2|V|-1$ imply that $G$ contains a unique $M$-circuit.  This $M$-circuit must be length-pure,
so length-rigid, and must contain all length edges of $G$. We can now
apply Theorem \ref{vspecial} to deduce that $(G,p)$ is not globally
rigid. This contradiction implies that $G$ is redundantly rigid. \eproof

We can use Theorem \ref{sparse} to give a simple characterisation of when a
generic realisation of a mixed graph $G$ with at most $2|V|-1$ edges is globally rigid.

\begin{theorem}\label{|D|+|L|= 2|V|-1cor}
Suppose $(G,p)$ is a generic realisation of a mixed graph
$G=(V;D,L)$ with $|D|+|L|\leq  2|V|-1$. Then $(G,p)$ is globally
rigid if and only if $G$ is rigid and either $|L|=1$, or $|D|+|L|=
2|V|-1$ and the subgraph of $G$ induced by the unique circuit in
$M(G)$ is mixed, direction-balanced, and contains $L$.
\end{theorem}
\bproof
We first prove sufficiency. Suppose that $G$ is rigid. If $|L|=1$ then $(G,p)$ is globally rigid by
 Lemma \ref{special}. So we may suppose further that
$|L|\geq 2$, $|D|+|L|=2|V|-1$, $L$ is contained in the unique $M$-circuit $H$ of $G$, and $H$ is mixed and direction-balanced.
Let $H=(V';D',L)$. Then $(H,p|_H)$ is globally rigid by \cite[Theorem 6.2]{JJ1}.
Since $H$ is mixed and is the unique $M$-circuit of $G$, $G$ is direction independent and $H$ is an induced subgraph of $G$.
We also have  $|D|+|L|= 2|V|-1$ and $|D'|+|L'|= 2|V'|-1$ so
$|D\sm D'|= 2|V\sm V'|$.
Hence $G$ is direction reducible to $H$ by Lemma \ref{dindred}.
Theorem \ref{reduce} now implies that $(G,p)$ is globally rigid.

We next prove necessity. Suppose that $(G,p)$ is globally rigid. Then $G$ is rigid. If $|L|=1$ then there is nothing more to prove so we may assume that
$|L|\geq 2$. Since  $(G,p)$ is globally rigid, each length edge of $G$
is contained in an $M$-circuit by Theorem \ref{hend}(a). Hence $|D|+|L|= 2|V|-1$ and $L$ is contained in the unique $M$-circuit $H$ of $G$. If $G$ is
direction irreducible then $G$ must be
direction-balanced and redundantly rigid by Theorem \ref{sparse}, so $G=H$ and $G$ is a direction balanced mixed $M$-circuit.
 Hence we may suppose that $G$ has a direction reduction to a
subgraph $G_1=(V_1;D_1,L)$. Since   each length edge of $G$
is contained in $H$, $G$ has no direction-pure $M$-circuits, and hence
$G$ is direction-independent.
By Lemma \ref{dindred}, $|D\sm D_1|= 2|V\sm V_1|$. Since $|D|+|L|= 2|V|-1$, we may deduce that
$|D_1|+|L|= 2|V_1|-1$.
By Theorem \ref{reduce}, $(G_1,p|_{G_1})$ is globally rigid.
We may now use induction to deduce that the unique $M$-circuit $H_1$ contained in
$G_1$ is mixed, direction-balanced, and contains $L$. Since $G_1$ is a subgraph of $G$ we have  $H_1=H$.
\eproof

\end{document}